\documentclass[twoside,reqno]{amsart}

\usepackage{latexsym}

\newcommand{\qdn}{\hspace*{-1.5mm}}
\newcommand{\qqdn}{\hspace*{-2.5mm}}
\newcommand{\xqdn}{\hspace*{-5.0mm}}
\newcommand{\xxqdn}{\hspace*{-10mm}}

\newcommand{\ffnk}[4]{\left[\qdn\ba{#1}#3\\[2mm]#4\ea{\!\bigg|\:#2}\right]}

\newcommand{\binm}{\binom}

\newcommand{\nnm}{\nonumber}
\newcommand{\be}{\begin{equation}}
\newcommand{\ee}{\end{equation}}
\newcommand{\ba}{\begin{array}}
\newcommand{\ea}{\end{array}}
\newcommand{\bmn}{\begin{eqnarray}}
\newcommand{\emn}{\end{eqnarray}}
\newcommand{\bnm}{\begin{eqnarray*}}
\newcommand{\enm}{\end{eqnarray*}}
\newcommand{\bln}{\begin{subequations}}
\newcommand{\eln}{\end{subequations}}

\newtheorem{thm}{Theorem}
\newtheorem{conjecture}[thm]{Conjecture}

\newtheorem{corl}[thm]{Corollary}
\newtheorem{prop}[thm]{Proposition}

\newtheorem{entry}{Entry}

\newcommand{\bbtm}[4]{\bibitem{kn:#1}{#2,}~{#3,}~{#4.}}
\newcommand{\cito}[1]{\cite{kn:#1}}
\newcommand{\citu}[2]{\cite[#2]{kn:#1}}

\begin{document} 

\title{Partial theta function identities from \\ Wang and Ma's conjecture}
\author{Chuanan Wei}


\dedicatory{Department of Medical Informatics\\
 Hainan Medical University, Haikou 571199, China}
\thanks{\emph{Email address}: weichuanan78@163.com}


\begin{abstract}
Recently, Wang and Ma propose a conjecture associated with the
possible generalization of Andrews-Warnaar identities. It is
confirmed in this paper. As the applications of this conjecture, we
prove that a family of series can be expressed by the partial theta
functions and construct some new partial theta function identities.
\end{abstract}

\maketitle\thispagestyle{empty}
\markboth{Chuanan Wei}
         { Partial theta functions identities}
\section{Introduction}

For an integer $n$ and two complex numbers $x$ and $q$ with $|q|<1$,
define the $q$-shifted factorial to be
 \bnm
(x;q)_{\infty}=\prod_{i=0}^{\infty}(1-xq^i),\quad
(x;q)_n=\frac{(x;q)_{\infty}}{(xq^n;q)_{\infty}}.
 \enm
For simplifying the expression, we shall adopt the following
notations:
 \bnm
&&(x_1,x_2,\cdots,x_r;q)_{\infty}=(x_1;q)_{\infty}(x_2;q)_{\infty}\cdots(x_r;q)_{\infty},\\
&&(x_1,x_2,\cdots,x_r;q)_{n}=(x_1;q)_{\infty}(x_2;q)_{\infty}\cdots(x_r;q)_{n}.
 \enm
Following Gasper and Rahman \cito{gasper}, define the basic
hypergeometric series by
 \bnm
{_{1+r}\phi_s}\ffnk{cccccc}{q;z}{a_0,&a_1,&\cdots,a_r}{&b_1,&\cdots,b_s}
  =\sum_{k=0}^{\infty}\frac{(a_0,a_1,\cdots,a_r;q)_k}{(q,b_1,\cdots,b_s;q)_k}
\Big\{(-1)^kq^{\binm{k}{2}}\Big\}^{s-r}z^k.
 \enm
Then a known terminating summation formula (cf. \citu{gasper}{p.
42}) can be stated as
 \bmn\label{phi-65}
  {_6\phi_5}\ffnk{cccccccccc}{q;\frac{aq^{n+1}}{bc}}
 {a,qa^{\frac{1}{2}},-qa^{\frac{1}{2}},b,c,q^{-n}}
 {a^{\frac{1}{2}},-a^{\frac{1}{2}},aq/b,aq/c,aq^{n+1}}=
\frac{(aq,aq/bc;q)_n}{(aq/b,aq/c;q)_n}.
 \emn

Sums of the form
 \[\sum_{n=-\infty}^{\infty}q^{An^2+Bn}x^n(A>0,x\neq0)\]
  are named theta functions. A nice result
is the famous Jacobi triple product identity (cf.
\citu{gasper}{p.15}):
  \bmn\label{jacobi}
 \quad\qquad\sum_{n=-\infty}^{\infty}(-1)^nq^{\binm{n}{2}}x^n=(q,x,q/x;q)_{\infty}.
  \emn
Sums of the form
 \bmn \label{theta}
  \sum_{n=0}^{\infty}q^{An^2+Bn}x^n(A>0)
  \emn
  are called partial theta functions because of the above fact (cf. \cite{kn:andrews-a,kn:andrews-b}).
   Partial theta function identities first appeared
in Ramanujan's lost notebook \cito{ramanujan}. Two beautiful results
from page 4 and page 12 of it are laid out as follows:
 \bnm
&&\xxqdn\sum_{n=0}^{\infty}q^{\frac{n(n+1)}{2}}a^n=\sum_{n=0}^{\infty}\frac{(q;q^2)_nq^n}{(a;q)_{n+1}(q/a;q)_n}
+\sum_{n=0}^{\infty}\frac{(-1)^{n+1}q^{n(n+1)}a^{2n+1}}{(-q,a,q/a;q)_{\infty}},\\
&&\xxqdn\sum_{n=0}^{\infty}q^{n(n+1)}a^n=\sum_{n=0}^{\infty}\frac{(q^{n+1};q)_nq^n}{(a;q)_{n+1}(q/a;q)_n}
-\sum_{n=0}^{\infty}\frac{q^{n(3n+2)}a^{3n+1}(1-aq^{2n+1})}{(a,q/a;q)_{\infty}}.
 \enm
More details on this kind of identities be found in Andrews and
Berndt \citu{andrews-c}{Chapter 6}.

An interesting partial theta function identity due to Warnaar
\citu{warnaar}{Theorem 1.5} reads
 \bmn\label{warnaar-a}
 \theta(q,a)+\theta(q,b)-1=(q,a,b;q)_{\infty}\sum_{n=0}^{\infty}\frac{(ab/q;q)_{2n}}{(q,a,b,ab;q)_n}q^n,
  \emn
where the symbol on the left-hand side is
 \[\theta(q,x)=\sum_{n=0}^{\infty}(-1)^nq^{\binm{n}{2}}x^n.\]
Obviously, \eqref{warnaar-a} is a generalization of \eqref{jacobi}.
Subsequently, Andrews and Warnaar \citu{andrews-d}{Theorem 1.1} gave
the product formula of partial theta functions:
\bmn\label{warnaar-b}
\theta(q,a)\theta(q,b)=(q,a,b;q)_{\infty}\sum_{n=0}^{\infty}\frac{(ab/q;q)_{2n}}{(q,a,b,ab/q;q)_n}q^n.
\emn
 Berkovich \cito{berkovich} showed that \eqref{warnaar-a} is equivalent to the
following result due to Schilling and Warnaar \citu{schilling}{Lemma
4.3}:
  \bmn\label{warnaar-c}
\frac{
\theta(q,a)-\theta(q,b)}{a-b}=-(q,aq,bq;q)_{\infty}\sum_{n=0}^{\infty}\frac{(ab;q)_{2n}}{(q,aq,bq,ab;q)_n}q^n.
  \emn
It should be mentioned that Alladi and Berkovich \cito{alladi}
derived another two-parameter generalization of \eqref{jacobi}:
 \bnm
&&\sum_{n=0}^{\infty}(-1)^nq^{\binm{n}{2}+2n}\bigg\{\frac{a}{a-1}\theta(q,aq^{1+n})+\frac{b}{b-1}\theta(q,bq^{1+n})\bigg\}\frac{(ab;q)_n}{(q;q)_n}\\
&&\:\:+\:\frac{1-ab}{(1-a)(1-b)}\sum_{n=0}^{\infty}(-1)^nq^{\binm{n}{2}+2n}\theta(q,q^{1+n})\frac{(ab;q)_n}{(q;q)_n}=(q,aq,bq;q)_{\infty}.
 \enm
For more formulas related to partial theta functions, the reader is
referred to the papers \cite{kn:berndt,kn:jo,kn:kostov,kn:ma-a}.

Recently, Wang and Ma \cito{ma-b} established the following theorem.

\begin{thm}\label{thm-a}
 Let $m,n$ be nonnegative integers and $a,b$ complex numbers. Then
 \[U_m(b)\theta(q,a)=(q,a,b;q)_{\infty}\sum_{n=0}^{\infty}\frac{(abq^{n-1};q)_n}{(q,a;q)_n}\frac{V_{m,n}(a,b)}{(b;q)_{m+n}}q^n,\]
where the signs $U_m(b)$ and $V_{m,n}(a,b)$ denote
 \bnm
 &&U_m(b)=\sum_{k=0}^{\infty}\frac{(q^{1-m};q)_k}{(q;q)_k(bq^k;q)_{m}}(1-bq^{2k})b^{2k}q^{2k^2-k+mk},\\
&&V_{m,n}(a,b)={_2\phi_1}\ffnk{cccccccccc}{q;bq^{m+n}}{q^{-m},q^{-n}}{abq^{n-1}}.
  \enm
\end{thm}
When $m=0$, it reduces to \eqref{warnaar-b}. In the same paper,
Theorem \ref{thm-a} was employed to deduce the following two
surprising results.

\begin{thm}\label{thm-b}
 Let $a$ and $b$ be complex numbers. Then
 \[\theta(q,a)=L(a,b)+bL(aq,bq),\]
where the notation on the right-hand side stands for
 \[L(a,b)=(q,a,b;q)_{\infty}\sum_{n=0}^{\infty}\frac{(ab/q^2;q)_{2n}}{(q,a,b,ab/q^2;q)_n}q^n.\]
\end{thm}

\begin{thm}\label{thm-c}
 Let $a$ and $b$ be complex numbers. Then
 \[\theta(q,a)=\frac{q}{q+b}P(a,b)+\frac{b(1+q)}{q+b}P(aq,bq)+\frac{b^2q}{q+b}P(aq^2,bq^2),\]
where the symbol on the right-hand side is
 \[P(a,b)=(q,a,b;q)_{\infty}\sum_{n=0}^{\infty}\frac{(ab/q^3;q)_{2n}}{(q,a,b,ab/q^3;q)_n}q^n.\]
\end{thm}
By expressing $L(a,b)$ in terms of $\theta(q,a)$ and $\theta(q,b)$,
Wang and Ma offered the directed proof of \eqref{warnaar-c}. The
fact leads us to consider that whether $P(a,b)$ can be expressed
through partial theta functions.

Based on Theorems \ref{thm-b} and \ref{thm-c},  Wang and Ma propose
the following conjecture related to the possible generalization of
Andrews-Warnaar identities at the end of their paper.

\begin{conjecture} \label{conjecture}
 For an integer $m\geq2$ and two complex numbers $a$ and $b$, define
$P_m(a,b)$ to be
\[P_m(a,b)=(q,a,b;q)_{\infty}\sum_{n=0}^{\infty}\frac{(ab/q^m;q)_{2n}}{(q,a,b,ab/q^m;q)_n}q^n.\]
Then $\theta(q,a)$ can be expressed as a linear combination of
$\{P_m(aq^k,bq^k)\}_{k=0}^{m-1}$, namely,
 \[\theta(q,a)=\sum_{k=0}^{m-1}\lambda_k(m,b)P_m(aq^k,bq^k),\]
where the coefficients $\lambda_k(m,b)$ are independent of $a$.
\end{conjecture}

If the conjecture is true, it is natural for us to ponder over that
whether $P_m(a,b)$ can be expressed via partial theta functions.

The structure of the paper is arranged as follows. We shall confirm
Conjecture \ref{conjecture} in Section 2. As the applications of
this conjecture, we prove that a family of series can be expressed
by partial theta functions and construct some new partial theta
function identities in Sections 3 and 4.
\section{Confirmation of Wang and Ma's conjecture}
\begin{thm}\label{thm-d}
For an integer $m\geq2$ and two complex numbers $a$ and $b$, there
holds
 \bnm
 \theta(q,a)=\sum_{k=0}^{m-1}\frac{(q^{m-k};q)_k(bq^{k-m+1})^k}{(q;q)_kQ_m(bq^{1-m})}P_m(aq^k,bq^k),
 \enm
where $P_m(a,b)$ has been given in Conjecture \ref{conjecture} and
 \[Q_m(b)=\sum_{i=0}^{m-2}\frac{(q^{2-m};q)_i}{(q;q)_i(bq^i;q)_{m-1}}(1-bq^{2i})b^{2i}q^{(2i-2+m)i}.\]
\end{thm}

\begin{proof}
Replacing $m$ by $m-1$ in Theorem \ref{thm-a}, we obtain
 \bnm
Q_m(b)\theta(q,a)&&\xqdn=(q,a,b;q)_{\infty}\\
  &&\xqdn\times\sum_{n=0}^{\infty}\frac{(abq^{n-1};q)_n}{(q,a;q)_n}\frac{q^n}{(b;q)_{m+n-1}}
 {_2\phi_1}\ffnk{cccccccccc}{q;bq^{m+n-1}}
 {q^{1-m},q^{-n}}{abq^{n-1}}.
 \enm
By means of the relation
\[(x;q)_k=(-1)^kq^{\binm{k}{2}}x^k(q^{1-k}/x;q)_k,\]
we can proceed as follows:
 \bnm
\xqdn Q_m(b)\theta(q,a)&&\xqdn=(q,a,b;q)_{\infty}\\
  &&\xqdn\times\sum_{n=0}^{\infty}\frac{(abq^{n-1};q)_n}{(q,a;q)_n}\frac{q^n}{(b;q)_{m+n-1}}
\sum_{k=0}^{m-1}\frac{(q^{1-m},q^{-n};q)_k}{(q,abq^{n-1};q)_k}(bq^{m+n-1})^k\\
&&\xqdn=(q,a,b;q)_{\infty}\\
  &&\xqdn\times\sum_{n=0}^{\infty}\frac{(abq^{n-1};q)_n}{(q,a;q)_n}\frac{q^n}{(b;q)_{m+n-1}}
\sum_{k=0}^{m-1}\frac{(q^{m-k},q^{1+n-k};q)_k}{(q,abq^{n-1};q)_k}(bq^{k-1})^k.
 \enm
Interchange the summation order to get
 \bnm
Q_m(b)\theta(q,a)&&\xqdn=(q,a,b;q)_{\infty}\\
  &&\xqdn\times\sum_{k=0}^{m-1}\frac{(q^{m-k};q)_k}{(q;q)_k}(bq^{k-1})^k
  \sum_{n=0}^{\infty}\frac{(abq^{n-1};q)_n}{(q,a;q)_n}\frac{q^n}{(b;q)_{m+n-1}}
\frac{(q^{1+n-k};q)_k}{(abq^{n-1};q)_k}\\
&&\xqdn=(q,a,bq^{m-1};q)_{\infty}\\
  &&\xqdn\times\sum_{k=0}^{m-1}\frac{(q^{m-k};q)_k}{(q;q)_k}(bq^{k-1})^k
  \sum_{n=0}^{\infty}\frac{(ab/q;q)_{2n}}{(q,a;q)_n}\frac{q^n}{(bq^{m-1};q)_{n}}
\frac{(q^{1+n-k};q)_k}{(ab/q;q)_{n+k}}\\
&&\xqdn=(q,a,bq^{m-1};q)_{\infty}\\
  &&\xqdn\times\sum_{k=0}^{m-1}\frac{(q^{m-k};q)_k}{(q;q)_k}(bq^{k-1})^k
  \sum_{n=k}^{\infty}\frac{(ab/q;q)_{2n}}{(q,a;q)_n}\frac{q^n}{(bq^{m-1};q)_{n}}
\frac{(q^{1+n-k};q)_k}{(ab/q;q)_{n+k}}.
 \enm
Shifting the summation index $n\to n+k$, the last equation can be
manipulated as
 \bnm
Q_m(b)\theta(q,a)
&&\xqdn=(q,a,bq^{m-1};q)_{\infty}\sum_{k=0}^{m-1}\frac{(q^{m-k};q)_k}{(q;q)_k}(bq^{k-1})^k
  \\&&\xqdn\times\sum_{n=0}^{\infty}\frac{(ab/q;q)_{2n+2k}}{(q,a;q)_{n+k}}\frac{q^{n+k}}{(bq^{m-1};q)_{n+k}}
\frac{(q^{1+n};q)_k}{(ab/q;q)_{n+2k}}
 \enm
 \bnm
&&\xqdn=(q,a,bq^{m-1};q)_{\infty}\sum_{k=0}^{m-1}\frac{(q^{m-k};q)_k}{(q;q)_k}(bq^{k})^k
  \\&&\xqdn\times\sum_{n=0}^{\infty}\frac{(ab/q^{1-2k};q)_{2n}}{(q;q)_n(a,bq^{m-1};q)_{n+k}}
   \frac{q^{n}}{(ab/q^{1-2k};q)_{n}}\\
&&\xqdn=\sum_{k=0}^{m-1}\frac{(q^{m-k};q)_k}{(q;q)_k}(bq^{k})^k(q,aq^k,bq^{k+m-1};q)_{\infty}
  \\&&\xqdn\times\sum_{n=0}^{\infty}\frac{(ab/q^{1-2k};q)_{2n}}{(q;q)_n(aq^k,bq^{k+m-1};q)_{n}}
   \frac{q^{n}}{(ab/q^{1-2k};q)_{n}}.
 \enm
Employ the substitution $b\to bq^{1-m}$ to gain
 \bnm
  Q_m(bq^{1-m})\theta(q,a)
&&\xqdn=\sum_{k=0}^{m-1}\frac{(q^{m-k};q)_k}{(q;q)_k}(bq^{k-m+1})^k(q,aq^k,bq^k;q)_{\infty}
  \\&&\xqdn\times\sum_{n=0}^{\infty}\frac{(ab/q^{m-2k};q)_{2n}}{(q,aq^k,bq^k,ab/q^{m-2k};q)_{n}}
   q^{n}.
 \enm
Dividing both sides by $Q_m(bq^{1-m})$, we achieve Theorem
\ref{thm-d}.
\end{proof}

It is evident that Theorem \ref{thm-d} becomes Theorem \ref{thm-b}
when $m=2$. Similarly, Theorem \ref{thm-d} reduces to Theorem
\ref{thm-c} when $m=3$. Substituting Theorem \ref{thm-d} into
\eqref{warnaar-b}, we attain the product formula.

\begin{corl}\label{corl-a}
For an integer $m\geq2$ and two complex numbers $a$ and $b$, there
holds
 \bnm
 &&\xqdn\sum_{k=0}^{m-1}\frac{(q^{m-k};q)_k}{(q;q)_k}(aq^{k-m+1})^kP_m(aq^k,bq^k)
 \sum_{k=0}^{m-1}\frac{(q^{m-k};q)_k}{(q;q)_k}(bq^{k-m+1})^kP_m(aq^k,bq^k)\\
 &&\xqdn\:\:=\,Q_m(aq^{1-m})Q_m(bq^{1-m})(q,a,b;q)_{\infty}
 \sum_{n=0}^{\infty}\frac{(ab/q;q)_{2n}}{(q,a,b,ab/q;q)_n}q^n.
 \enm
\end{corl}

When $m=2$, Corollary \ref{corl-a} becomes the simple result:
  \bnm
  &&\big\{L(a,b)+aL(aq,bq)\big\}\big\{L(a,b)+bL(aq,bq)\big\}\\
  &&\:=(q,a,b;q)_{\infty}\sum_{n=0}^{\infty}\frac{(ab/q;q)_{2n}}{(q,a,b,ab/q;q)_n}q^n.
  \enm

\section{The applications of Wang and Ma's conjecture}
\begin{thm}\label{thm-e}
Let $m\geq2$ be an integer and $a$, $b$ be complex numbers. Then
$P_m(a,b)$ can be expressed as the linear combination of
$\{\theta(q,a/q^k)\}_{k=0}^{m-2}$ and
$\{\theta(q,b/q^k)\}_{k=0}^{m-2}$.
\end{thm}

\begin{proof}
Theorem \ref{thm-d} can be written as
 \bmn \label{theta-a}
 \theta(q,a)=\sum_{k=0}^{m-1}\lambda_k(m,b)P_m(aq^k,bq^k),
 \emn
where the coefficients on the right-hand side denote
\[\lambda_k(m,b)=\frac{(q^{m-k};q)_k(bq^{k-m+1})^k}{(q;q)_kQ_m(bq^{1-m})},\quad k=0,1,\cdots,m-1.\]
Interchange the parameters $a$ and $b$ to obtain
 \bmn
\label{theta-b}
 \theta(q,b)=\sum_{k=0}^{m-1}\lambda_k(m,a)P_m(aq^k,bq^k).
 \emn
The difference of
$\eqref{theta-a}\times\lambda_{m-1}(m,a)-\eqref{theta-b}\times\lambda_{m-1}(m,b)$
produces
 \bmn \label{theta-c}
&&\xqdn\lambda_{m-1}(m,a)\theta(q,a)-\lambda_{m-1}(m,b)\theta(q,b)
\nnm\\
&&\xqdn\:=\:\sum_{k=0}^{m-2}\big\{\lambda_{m-1}(m,a)\lambda_k(m,b)-\lambda_{m-1}(m,b)\lambda_k(m,a)\big\}P_m(aq^k,bq^k).
 \emn
The difference of
$\eqref{theta-a}\times\lambda_{0}(m,a)-\eqref{theta-b}\times\lambda_{0}(m,b)$
creates
 \bnm
&&\xqdn\xxqdn\lambda_{0}(m,a)\theta(q,a)-\lambda_{0}(m,b)\theta(q,b)\\
&&\xqdn\xxqdn\:=\:\sum_{k=1}^{m-1}\big\{\lambda_{0}(m,a)\lambda_k(m,b)-\lambda_{0}(m,b)\lambda_k(m,a)\big\}P_m(aq^k,bq^k).
 \enm
Perform the replacements $a\to a/q$, $b\to b/q$ in the last equation
to get
 \bmn\label{theta-d}
&&\xxqdn\lambda_{0}(m,a/q)\theta(q,a/q)-\lambda_{0}(m,b/q)\theta(q,b/q)
\nnm\\
&&\xxqdn\:=\:\sum_{k=0}^{m-2}\big\{\lambda_{0}(m,a/q)\lambda_{k+1}(m,b/q)-\lambda_{0}(m,b/q)\lambda_{k+1}(m,a/q)\big\}P_m(aq^k,bq^k).
 \emn
Define two signs by
 \bnm
&&\qqdn\xqdn\alpha(m,a,b)=\lambda_{0}(m,a/q)\lambda_{m-1}(m,b/q)-\lambda_{0}(m,b/q)\lambda_{m-1}(m,a/q),\\
&&\qqdn\xqdn\beta(m,a,b)=\lambda_{m-1}(m,a)\lambda_{m-2}(m,b)-\lambda_{m-1}(m,b)\lambda_{m-2}(m,a).
\enm
  Then the difference
$\eqref{theta-c}\times\alpha(m,a,b)-\eqref{theta-d}\times\beta(m,a,b)$
brings out
 \bnm
&&\alpha(m,a,b)\big\{\lambda_{m-1}(m,a)\theta(q,a)-\lambda_{m-1}(m,b)\theta(q,b)\big\}\\
&&\:-\:\beta(m,a,b)\big\{\lambda_{0}(m,a/q)\theta(q,a/q)-\lambda_{0}(m,b/q)\theta(q,b/q)\big\}\\
&&\:=\sum_{k=0}^{m-3}\gamma_k(m,a,b)P_m(aq^k,bq^k),
 \enm
where the coefficients on the right-hand stand for
 \bnm
&&\xxqdn\gamma_k(m,a,b)=\alpha(m,a,b)\big\{\lambda_{m-1}(m,a)\lambda_k(m,b)-\lambda_{m-1}(m,b)\lambda_k(m,a)\big\}\\
&&\qquad-\,\beta(m,a,b)\big\{\lambda_{0}(m,a/q)\lambda_{k+1}(m,b/q)-\lambda_{0}(m,b/q)\lambda_{k+1}(m,a/q)\big\},\\
&&\xxqdn k=0,1,\cdots,m-3.
 \enm
Continue the above process. Finally, only the term $P_m(a,b)$ is
holden back on the right-hand side. Thus $P_m(a,b)$ is expressed as
the linear combination of $\{\theta(q,a/q^k)\}_{k=0}^{m-2}$ and
$\{\theta(q,b/q^k)\}_{k=0}^{m-2}$.
\end{proof}

The proof of Theorem \ref{thm-e} offers us the method to express
$P_m(a,b)$ through $\{\theta(q,a/q^k)\}_{k=0}^{m-2}$ and
$\{\theta(q,b/q^k)\}_{k=0}^{m-2}$. Utilizing this method, it is easy
to derive \eqref{warnaar-c} from Theorem \ref{thm-b}. Further, we
can deduce the following partial theta function identity from
Theorem \ref{thm-c}.

\begin{thm}\label{thm-f}
 Let $a$ and $b$ be complex numbers. Then
 \bnm
U(a,b)-U(b,a)=\frac{(a-b)(a-bq)(aq-b)}{ab(a+b)(1+q)}P(a,b),
 \enm
where $P(a,b)$ has been furnished in Theorem \ref{thm-c} and the
notation $U(a,b)$ is
 \bnm
U(a,b)=\frac{a(b+q)}{b(1+q)}\theta(q,a)+\frac{b+q^2}{a+b}\theta(q,a/q).
 \enm
\end{thm}

\begin{thm}\label{thm-g}
 Let $a$ and $b$ be complex numbers. Then
 \bnm
&&\qqdn\frac{abq+aq^2-b^2-bq^3}{q(a-b)b^2}\theta(q,a)-\frac{abq+bq^2-a^2-aq^3}{q(a-b)a^2}\theta(q,b)-\frac{1+q}{ab}\\
&&\qqdn\:\:=-\frac{(a-bq)(b-aq)}{a^2b^2}P(a,b).
 \enm
\end{thm}

\begin{proof}
First of all, we have the relation
 \bnm
&&\frac{a(b+q)}{b(1+q)}\theta(q,a)+\frac{b+q^2}{a+b}\theta(q,a/q)\\
&&\:=\frac{a(b+q)}{b(1+q)}\theta(q,a)+\frac{b+q^2}{a+b}\bigg\{1+\sum_{n=1}^{\infty}(-1)^nq^{\binm{n}{2}}\bigg(\frac{a}{q}\bigg)^n\bigg\}\\
&&\:=\frac{a(b+q)}{b(1+q)}\theta(q,a)+\frac{b+q^2}{a+b}\bigg\{1+\sum_{n=0}^{\infty}(-1)^{n+1}q^{\binm{n+1}{2}}\bigg(\frac{a}{q}\bigg)^{n+1}\bigg\}\\
&&\:=\frac{a(b+q)}{b(1+q)}\theta(q,a)+\frac{b+q^2}{a+b}\bigg\{1-\frac{a}{q}\theta(q,a)\bigg\}\\
&&\:=\frac{a(abq+aq^2-b^2-bq^3)}{qb(1+q)(a+b)}\theta(q,a)+\frac{b+q^2}{a+b}.
 \enm
Subsequently, we gain another relation
 \bnm
&&\qdn\qqdn\xxqdn\xxqdn\frac{b(a+q)}{a(1+q)}\theta(q,b)+\frac{a+q^2}{a+b}\theta(q,b/q)\\
&&\qdn\qqdn\xxqdn\xxqdn\:=\frac{b(abq+bq^2-a^2-aq^3)}{qa(1+q)(a+b)}\theta(q,b)+\frac{a+q^2}{a+b}.
 \enm
Substituting the last two relations into Theorem \ref{thm-f}, we
achieve Theorem \ref{thm-g}.
\end{proof}

\begin{corl}[$b=\frac{a^2+aq^3}{aq+q^2}$ in Theorem \ref{thm-g}]\label{corl-b}
 \bnm
\qquad\theta(q,a)=\frac{q(a+q^3)}{a^2+aq^2+aq^3+q^4}-\frac{aq^2(1+q)(1-q)^2}{(a+q)(a^2+aq^2+aq^3+q^4)}P\bigg(a,\frac{a^2+aq^3}{aq+q^2}\bigg).
 \enm
\end{corl}

\begin{thm}\label{thm-h}
 Let $a$ and $b$ be complex numbers. Then
 \bnm
&&\frac{q^2}{1+q}\Omega(a/q,b/q)-\frac{ab(a+b-q-q^2)}{(1+q)(a+b)}\Omega(a,b)-\frac{a^2b^2}{q(a+b)}\Omega(aq,bq)\\
&&\:-\frac{aq^2+bq^2-ab-abq}{(1+q)(a+b)}=\frac{(a-bq)(b-aq)}{(1+q)(a+b)}P(a,b),
 \enm
where the symbol on the left-hand side denotes
\[\Omega(a,b)=\sum_{n=0}^{\infty}(-1)^nq^{\binm{n}{2}}b^n\theta(q,aq^n).\]
\end{thm}

\begin{proof}
Above all, we have the relation
 \bnm
&&\frac{a(b+q)}{b(1+q)}\theta(q,a)-\frac{b(a+q)}{a(1+q)}\theta(q,b)\\
&&\:=\sum_{n=0}^{\infty}(-1)^nq^{\binm{n}{2}}\frac{ab(a^{n+1}-b^{n+1})+q(a^{n+2}-b^{n+2})}{ab(1+q)}
\enm
\bnm
&&\:=\frac{a-b}{1+q}\sum_{n=0}^{\infty}(-1)^nq^{\binm{n}{2}}\frac{a^{n+1}-b^{n+1}}{a-b}\\
&&\:+\:\frac{q(a-b)}{ab(1+q)}\sum_{n=0}^{\infty}(-1)^nq^{\binm{n}{2}}\frac{a^{n+2}-b^{n+2}}{a-b}\\
&&\:=\frac{a-b}{1+q}\sum_{n=0}^{\infty}(-1)^nq^{\binm{n}{2}}b^n\theta(q,aq^n)\\
&&\:-\:\frac{q^2(a-b)}{ab(1+q)}\sum_{n=0}^{\infty}(-1)^nq^{\binm{n}{2}}(b/q)^n\theta(q,aq^{n-1})+\frac{q^2(a-b)}{ab(1+q)}.
 \enm
In the next place, we attain another relation
 \bnm
&&\xxqdn\frac{b+q^2}{a+b}\theta(q,a/q)-\frac{a+q^2}{a+b}\theta(q,b/q)\\
&&\xxqdn\:=\sum_{n=0}^{\infty}(-1)^nq^{\binm{n}{2}-n}\frac{ab(a^{n-1}-b^{n-1})+q^2(a^{n}-b^{n})}{a+b}\\
&&\xxqdn\:=\frac{ab(a-b)}{a+b}\sum_{n=0}^{\infty}(-1)^nq^{\binm{n}{2}-n}\frac{a^{n-1}-b^{n-1}}{a-b}\\
&&\xxqdn\:+\:\frac{q^2(a-b)}{a+b}\sum_{n=0}^{\infty}(-1)^nq^{\binm{n}{2}-n}\frac{a^{n}-b^{n}}{a-b}\\
&&\xxqdn\:=\frac{ab(a-b)}{q(a+b)}\sum_{n=0}^{\infty}(-1)^nq^{\binm{n}{2}}(bq)^n\theta(q,aq^{n+1})\\
&&\xxqdn\:-\:\frac{q(a-b)}{a+b}\sum_{n=0}^{\infty}(-1)^nq^{\binm{n}{2}}b^n\theta(q,aq^n)-\frac{a-b}{a+b}.
 \enm
Substituting the last two relations into Theorem \ref{thm-f}, we
obtain Theorem \ref{thm-h}.
\end{proof}

\begin{corl}[$a\to aq,b\to q+q^2-aq$ in Theorem \ref{thm-h}]\label{corl-c}
 \bnm
&&\qdn\qqdn\xxqdn\Omega(a,1+q-a)-a^2(1+q-a)^2\Omega(aq^2,q^2+q^3-aq^2)-(1-a)(1-a/q)\\
&&\qdn\qqdn\xxqdn\:=-(1+q)(1-a)(1-a/q)P(aq,q+q^2-aq).
 \enm
\end{corl}

\begin{thm}\label{thm-i}
 Let $a$, $b$, $c$ and $d$ be complex numbers. Then
 \bnm
&&\xqdn\sum_{k=0}^{\infty}\frac{1-abq^{2k}}{1-ab}\frac{(ab,c,d;q)_k}{(q,abq/c,abq/d;q)_k}q^{\binm{k}{2}+2k}\bigg(-\frac{ab}{cd}\bigg)^k \Theta_k(q,a,b)\\
&&\xqdn\:=\:\frac{(a-b)(a-bq)(b-aq)}{qab(a+b)(1+q)}(q,aq^2,bq^2;q)_{\infty}\sum_{n=0}^{\infty}\frac{(abq;q)_{2n}(abq/cd;q)_{n}}{(q,aq^2,bq^2,abq/c,abq/d;q)_n}q^n,
 \enm
where the sign on the left-hand side stands for
  \bnm
 \Theta_k(q,a,b)&&\xqdn\!=\frac{b(1+aq^{k+1})}{a(1+q)}\theta(q,bq^{k+2})-\frac{a(1+bq^{k+1})}{b(1+q)}\theta(q,aq^{k+2})\\
&&\xqdn\!+\:\frac{1+aq^{k}}{(a+b)q^{k+1}}\theta(q,bq^{k+1})-\frac{1+bq^{k}}{(a+b)q^{k+1}}\theta(q,aq^{k+1}).
 \enm
\end{thm}

\begin{proof}
According to the definition of $P(a,b)$, we get
 \bnm
&&\sum_{k=0}^{\infty}\frac{1-abq^{2k}}{1-ab}\frac{(ab,c,d;q)_k}{(q,abq/c,abq/d;q)_k}q^{\binm{k}{2}+2k}\bigg(-\frac{ab}{cd}\bigg)^kP(aq^{k+2},bq^{k+2})\\
&&\:=\:\sum_{k=0}^{\infty}\frac{1-abq^{2k}}{1-ab}\frac{(ab,c,d;q)_k}{(q,abq/c,abq/d;q)_k}q^{\binm{k+1}{2}}\bigg(-\frac{ab}{cd}\bigg)^k\\
&&\:\times\:\:(q,aq^{k+2},bq^{k+2};q)_{\infty}\sum_{n=0}^{\infty}\frac{(abq^{2k+1};q)_{2n}}{(q,aq^{k+2},bq^{k+2},abq^{2k+1};q)_n}q^{n+k}\\
&&\:=\:\sum_{k=0}^{\infty}\frac{1-abq^{2k}}{1-ab}\frac{(ab,c,d;q)_k}{(q,abq/c,abq/d;q)_k}q^{\binm{k+1}{2}}\bigg(-\frac{ab}{cd}\bigg)^k\\
&&\:\times\:\:(q,aq^{k+2},bq^{k+2};q)_{\infty}\sum_{n=k}^{\infty}\frac{(abq^{2k+1};q)_{2n-2k}}{(q,aq^{k+2},bq^{k+2},abq^{2k+1};q)_{n-k}}q^{n}\\
&&\:=\:\sum_{k=0}^{\infty}\frac{1-abq^{2k}}{1-ab}\frac{(ab,c,d;q)_k}{(q,abq/c,abq/d;q)_k}q^{\binm{k+1}{2}}\bigg(-\frac{ab}{cd}\bigg)^k\\
&&\:\times\:\:(q,aq^{2},bq^{2};q)_{\infty}\sum_{n=k}^{\infty}\frac{q^n}{(q;q)_{n-k}(abq;q)_{n+k}}\frac{(abq;q)_{2n}}{(aq^2,bq^2;q)_{n}}\\
&&\:=\:(q,aq^2,bq^2;q)_{\infty}\sum_{n=0}^{\infty}\frac{(abq;q)_{2n}q^n}{(q,aq^2,bq^2,abq;q)_n}\\
&&\:\times\:\:{_6\phi_5}\ffnk{cccccccccc}{q;\frac{abq^{n+1}}{cd}}
 {ab,q(ab)^{\frac{1}{2}},-q(ab)^{\frac{1}{2}},c,d,q^{-n}}{(ab)^{\frac{1}{2}},-(ab)^{\frac{1}{2}},abq/c,abq/d,abq^{n+1}}.
 \enm
Calculate the series on the right-hand side by \eqref{phi-65} to
gain
 \bnm
&&\sum_{k=0}^{\infty}\frac{1-abq^{2k}}{1-ab}\frac{(ab,c,d;q)_k}{(q,abq/c,abq/d;q)_k}q^{\binm{k}{2}+2k}\bigg(-\frac{ab}{cd}\bigg)^kP(aq^{k+2},bq^{k+2})\\
&&\:=\:(q,aq^2,bq^2;q)_{\infty}\sum_{n=0}^{\infty}\frac{(abq;q)_{2n}(abq/cd;q)_{n}}{(q,aq^2,bq^2,abq/c,abq/d;q)_n}q^n.
 \enm
Combining Theorem \ref{thm-f} with the last equation, we achieve
Theorem \ref{thm-i}.
\end{proof}

\begin{corl}[$cd=qab$ in Theorem \ref{thm-i}]\label{corl-d}
  \bnm
&&\xqdn\sum_{k=0}^{\infty}(-1)^k\frac{1-abq^{2k}}{1-ab}\frac{(ab;q)_k}{(q;q)_k}q^{\binm{k+1}{2}}\Theta_k(q,a,b)\\
&&\xqdn\:=\,\frac{(a-b)(a-bq)(b-aq)}{qab(a+b)(1+q)}(q,aq^2,bq^2;q)_{\infty}.
 \enm
\end{corl}

Employing the substitution $c\to abq/c$ in Theorem \ref{thm-i} and
then letting $b\to 0$, we attain the following formula after the
replacements $a\to a/q^2$ and $d\to b$.

\begin{prop}\label{prop-a}
 Let $a$ and $b$ be complex numbers. Then
 \bnm
 \qquad \sum_{k=0}^{\infty}\frac{(b;q)_k}{(q,c;q)_k}q^{\binm{k+1}{2}}\bigg(-\frac{c}{b}\bigg)^k\theta(q,aq^k)
=(q,a;q)_{\infty}\sum_{n=0}^{\infty}\frac{(c/b;q)_{n}}{(q,a,c;q)_n}q^n.
 \enm
\end{prop}

\begin{corl}[$c\to \infty$ in Proposition \ref{prop-a}]\label{corl-e}
  \bnm
 \qquad \sum_{k=0}^{\infty}\frac{(b;q)_k}{(q;q)_k}\bigg(\frac{q}{b}\bigg)^k\theta(q,aq^k)
=(q,a;q)_{\infty}\sum_{n=0}^{\infty}\frac{1}{(q,a;q)_n}\bigg(\frac{q}{b}\bigg)^n.
 \enm
\end{corl}

The case $b\to0$ of Corollary \ref{corl-d} and the case $b\to\infty$
of Corollary \ref{corl-e} are the same. It can be stated as
 \bnm
\quad\sum_{k=0}^{\infty}(-1)^k\frac{q^{\binm{k+1}{2}}}{(q;q)_k}\theta(q,aq^k)=(q,a;q)_{\infty}.
 \enm
\section{More partial theta function identities}

\begin{thm}\label{thm-k}
 Let $a$ and $b$ be complex numbers. Then
 \bnm
&&\xxqdn\frac{a^2+2aq+aq^2-bq^2+ab}{b(a-bq)(b-aq)}\theta(q,a)-\frac{(a+q)(a+b)}{a(a-bq)(b-aq)}\theta(q,b)\\
&&\xxqdn\:+\:\frac{(1+q)(b+q^2)}{b(a-bq)(b-aq)}\theta(q,a/q)-\frac{a(1+q)(b+q^2)+q(a^2-b^2)}{ab(a-bq)(b-aq)}\theta(q,b/q)\\
&&\xxqdn\:=-\frac{a-b}{qab}S(a,b),
 \enm
where the notation on the right-hand side is
 \[S(a,b)=(q,a,b;q)_{\infty}\sum_{n=0}^{\infty}\frac{(ab/q^3;q)_{2n}}{(q,a,b,ab/q^3;q)_n}q^{2n}.\]
\end{thm}

\begin{proof}
Performing the replacements $a\to a/q, b\to b/q$ in
\eqref{warnaar-c}, we obtain the following result (cf.
\citu{ma-b}{Corollary 3.1}):
 \bmn\label{warnaar-d}
\qdn
L(a,b)&&\xqdn\!=-\frac{q}{a-b}\theta(q,a/q)+\frac{q}{a-b}\theta(q,b/q)
 \nnm\\\nnm
 &&\xqdn\!=-\frac{q}{a-b}\bigg\{1-\frac{a}{q}\theta(q,a)\bigg\}+\frac{q}{a-b}\bigg\{1-\frac{b}{q}\theta(q,b)\bigg\}\\
 &&\xqdn\!=\frac{a}{a-b}\theta(q,a)-\frac{b}{a-b}\theta(q,b).
  \emn
It is interesting to see the relation
 \bnm
S(a,b)&&\xqdn\!=(q,a,b;q)_{\infty}\sum_{n=0}^{\infty}\frac{(ab/q^3;q)_{2n}}{(q,a,b,ab/q^3;q)_n}q^{n}\\
 &&\xqdn\!\times\:\bigg\{\frac{q}{b}-\frac{q}{b}(1-bq^{n-1})\bigg\}\\
&&\xqdn\!=\frac{q}{b}(q,a,b;q)_{\infty}\sum_{n=0}^{\infty}\frac{(ab/q^3;q)_{2n}}{(q,a,b,ab/q^3;q)_n}q^{n}\\
&&\xqdn\!-\:\frac{q}{b}(q,a,b;q)_{\infty}\sum_{n=0}^{\infty}\frac{(ab/q^3;q)_{2n}}{(q,a,b,ab/q^3;q)_n}(1-bq^{n-1})q^{n}\\
&&\xqdn\!=\frac{q}{b}P(a,b)-\frac{q}{b}L(a,b/q).
 \enm
Substituting Theorem \ref{thm-f} and \eqref{warnaar-d} into the last
equation, we get Theorem \ref{thm-k}.
\end{proof}

\begin{thm}\label{thm-l}
 Let $a$ and $b$ be complex numbers. Then
 \bnm
&&\:\xxqdn\frac{ab-a^2q-aq^2+bq^3}{b(a-bq)(b-aq)}\theta(q,a)-\frac{ab-b^2q-bq^2+aq^3}{a(a-bq)(b-aq)}\theta(q,b)\\
&&\:\xxqdn\:+\:\frac{q^2(a^2-b^2)}{ab(a-bq)(b-aq)}
=\frac{a-b}{ab}S(a,b).
 \enm
\end{thm}

\begin{proof}
First of all, we have the relation
 \bnm
&&\xqdn\frac{a^2+2aq+aq^2-bq^2+ab}{b(a-bq)(b-aq)}\theta(q,a)+\frac{(1+q)(b+q^2)}{b(a-bq)(b-aq)}\theta(q,a/q)\\
&&\xqdn\:=\frac{a^2+2aq+aq^2-bq^2+ab}{b(a-bq)(b-aq)}\theta(q,a)+\frac{(1+q)(b+q^2)}{b(a-bq)(b-aq)}\bigg\{1-\frac{a}{q}\theta(q,a)\bigg\}\\
&&\xqdn\:=\frac{a^2q+aq^2-bq^3-ab}{qb(a-bq)(b-aq)}\theta(q,a)+\frac{(1+q)(b+q^2)}{b(a-bq)(b-aq)}.
 \enm
Subsequently, we gain another relation
 \bnm
&&\qdn\frac{(a+q)(a+b)}{a(a-bq)(b-aq)}\theta(q,b)+\frac{a(1+q)(b+q^2)+q(a^2-b^2)}{ab(a-bq)(b-aq)}\theta(q,b/q)\\
&&\qdn\:=\frac{(a+q)(a+b)}{a(a-bq)(b-aq)}\theta(q,b)+\frac{a(1+q)(b+q^2)+q(a^2-b^2)}{ab(a-bq)(b-aq)}\bigg\{1-\frac{b}{q}\theta(q,b)\bigg\}\\
&&\qdn\:=\frac{b^2q+bq^2-aq^3-ab}{qa(a-bq)(b-aq)}\theta(q,b)+\frac{a(1+q)(b+q^2)+q(a^2-b^2)}{ab(a-bq)(b-aq)}.
 \enm
Substituting the last two relations into Theorem \ref{thm-k}, we
achieve Theorem \ref{thm-l}.
\end{proof}

\begin{corl}[$b=\frac{a^2q+aq^2}{a+q^3}$ in Theorem \ref{thm-l}]\label{corl-f}
 \bnm
\qqdn\theta\bigg(q,\frac{a^2q+aq^2}{a+q^3}\bigg)&&\xqdn\!=\frac{(a+q^2)(a+q^3)}{(a+q)(a^2+aq+aq^2+q^4)}\\
&&\xqdn\!-\,\frac{a^2(1-q)^2(1+q)}{(a+q)(a^2+aq+aq^2+q^4)}S\bigg(a,\frac{a^2q+aq^2}{a+q^3}\bigg).
 \enm
\end{corl}

Corollaries \ref{corl-b} and \ref{corl-f} are both strange partial
theta function identities. However, they are different with each
other.

In accordance with the proof of Theorem \ref{thm-e}, it is routine
to express $P_4(a,b)$ via $\{\theta(q,a/q^k)\}_{k=0}^2$ and
$\{\theta(q,b/q^k)\}_{k=0}^2$. The result is displayed in the
following theorem.

\begin{thm}\label{thm-m}
 Let $a$ and $b$ be complex numbers. Then
 \bnm
V(a,b)-V(b,a)=\frac{(a-bq)(a-bq^2)(aq-b)(aq^2-b)}{a^3b^3q(a+b)(1+q+q^2)}P_4(a,b),
 \enm
where the symbol on the left-hand side denotes
  \bnm
\xqdn V(a,b)&&\xqdn\!=\frac{(b^2+b+bq+q^2)\{ab+abq^2+(a+b)^2q\}}{b^3q(a+b)(1+q+q^2)}\theta(q,a)\\
 &&\xqdn\!+\:\frac{(b^2+bq+bq^2+q^4)(a+aq+aq^2+bq)}{ab^2q^2(a+b)}\theta(q,a/q)\\
 &&\xqdn\!+\:\frac{b^2+bq^2+bq^3+q^6}{a^2b^2}\theta(q,a/q^2).
 \enm
\end{thm}

\begin{corl}\label{corl-g}
 Let $a$ be a complex number. Then
 \bnm
&&\qqdn\xqdn\xxqdn\frac{a^2-aq-aq^3+q^5}{2q^2(1+q+q^2)}\theta(q,a)+\frac{a^2+aq+aq^3+q^5}{2q^2(1+q+q^2)}\theta(q,-a)+1\\
&&\qqdn\xqdn\xxqdn\:=\:\frac{(1+q)(1+q^2)}{1+q+q^2}P_4(a,-a).
 \enm
\end{corl}

\begin{proof}
The case $b=-a$ of Theorem \ref{thm-m} reads
 \bmn\label{relation}
&&\frac{a-a^2+aq-q^2}{1+q+q^2}\theta(q,a)+\frac{aq-a^2+aq^2-q^4}{aq}\theta(q,a/q)
\nnm\\\nnm
&&-\frac{a+a^2+aq+q^2}{1+q+q^2}\theta(q,-a)+\frac{aq+a^2+aq^2+q^4}{aq}\theta(q,-a/q)\\
&&=\frac{2(1+q)^2(1+q^2)}{1+q+q^2}P_4(a,-a)
 \emn
Above all, we have the relation
 \bnm
&&\frac{a-a^2+aq-q^2}{1+q+q^2}\theta(q,a)+\frac{aq-a^2+aq^2-q^4}{aq}\theta(q,a/q)\\
&&=\frac{a-a^2+aq-q^2}{1+q+q^2}\theta(q,a)+\frac{aq-a^2+aq^2-q^4}{aq}\bigg\{1-\frac{a}{q}\theta(q,a)\bigg\}\\
&&=\frac{(1+q)(a^2-aq-aq^3+q^5)}{q^2(1+q+q^2)}\theta(q,a)+\frac{aq-a^2+aq^2-q^4}{aq}.
 \enm
In the next place, we attain another relation
 \bnm
 &&-\frac{a+a^2+aq+q^2}{1+q+q^2}\theta(q,-a)+\frac{aq+a^2+aq^2+q^4}{aq}\theta(q,-a/q)\\
&&=\frac{(1+q)(a^2+aq+aq^3+q^5)}{q^2(1+q+q^2)}\theta(q,-a)+\frac{aq+a^2+aq^2+q^4}{aq}.
 \enm
Substituting the last two relations into \eqref{relation}, we obtain
Corollary \ref{corl-g}.
\end{proof}

 \textbf{Acknowledgments}

 The work is supported by the National Natural Science Foundation of China (No. 11661032).


\end{document}